\documentclass[reqno]{amsart}
\usepackage[psamsfonts]{amssymb}
\usepackage{amsmath, amsthm, amscd, amsfonts}
\usepackage{amsmath}
\usepackage{amssymb}
\usepackage{enumerate}
\usepackage{geometry}

\def\bege{\begin{equation}}
\def\ende{\end{equation}}

\def\begr{\begin{eqnarray}}
\def\endr{\end{eqnarray}}

\def\bege{\begin{equation}}
\def\ende{\end{equation}}
\def\begr{\begin{eqnarray}}
\def\endr{\end{eqnarray}}
\def\bnum{\begin{enumerate}}
\def\enum{\end{enumerate}}
\newtheorem{theorem}{Theorem}

\newtheorem{definition}[theorem]{Definition}

\allowdisplaybreaks[4]

\input{tcilatex}

\begin{document}
\title[Periodic solutions of a system of difference equations]{{\large 
\textbf{Periodic solutions of a system of nonlinear difference equations
with periodic coefficients}}}
\author[Durhasan Turgut Tollu]{Durhasan Turgut Tollu}
\address{Necmettin Erbakan University, Faculty of Science, Department of
Mathematics-Computer, Konya, Turkey}
\email{dttollu@erbakan.edu.tr}
\maketitle

\begin{abstract}
In this paper it is dealt with the following system of difference equations 
\begin{equation*}
x_{n+1}=\frac{a_{n}}{x_{n}}+\frac{b_{n}}{y_{n}},\ y_{n+1}=\frac{c_{n}}{x_{n}}%
+\frac{d_{n}}{y_{n}},\ n\in \mathbb{N}_{0},
\end{equation*}%
where the initial values $x_{0},y_{0}$ are positive real numbers and the
coefficients $\left( a_{n}\right) _{n\geq 0}$, $\left( b_{n}\right) _{n\geq
0}$, $\left( c_{n}\right) _{n\geq 0}$, $\left( d_{n}\right) _{n\geq 0}$ are
two-periodic sequences with positive terms. The system is an extention of a
system that every positive solution is two-periodic or converges to its a
two-periodic solution. Here, the long-term behavior of posistive solutions
of the system is examined by using a new method to solve the system.
\end{abstract}


\section{Introduction}

Studying concrete nonlinear difference equations and systems have attracted
a great recent interest. Particularly, there have been a renewed interest in
solvable nonlinear difference equations and systems for fifteen years (see,
e.g., \cite{Cinar}-\cite{Grove0}, \cite{Haddad},\cite{KaraTJM}, \cite%
{McGrath}-\cite{S8},\cite{T2}-\cite{YY6} and the related references
therein). Solvable difference equations are not interesting for themselves
only, but they can be also applied in other areas of mathematics, as well as
other areas of science (see, e.g., \cite{Grove},\cite{ll}).

One of the first examples of solvable nonlinear difference equations is
presented in note \cite{Brand} where Brand solves the nonlinear difference
equation%
\begin{equation}
x_{n+1}=\frac{ax_{n}+b}{cx_{n}+d},\ n\in \mathbb{N}_{0},  \label{b}
\end{equation}%
where the initial value $x_{0}$ is a real number and the parameters $a$, $b$%
, $c$, $d$ are real numbers with the restrictions $c\neq 0$, $ad-bc\neq 0$,
and studies long-term behavior of solutions the equation. The note presents
a transformation which transforms the nonlinear equation into a linear one.
The idea has been used many times in showing solvability of some difference
equations, as well as of some systems of difference equations (see, e.g., 
\cite{Haddad},\cite{KaraTJM},\cite{McGrath},\cite{HS},\cite{T2}-\cite{Tollu},%
\cite{YY6}). Another example of solvable nonlinear difference equations is
the following system of nonlinear difference equation%
\begin{equation}
x_{n+1}=\frac{a}{x_{n}}+\frac{b}{y_{n}},\ y_{n+1}=\frac{c}{x_{n}}+\frac{d}{%
y_{n}},\ n\in \mathbb{N}_{0},  \label{b1}
\end{equation}%
where the initial value $x_{0}$, $y_{0}$ are positive real numbers and the
parameters $a$, $b$, $c$, $d$ are positive real numbers. System (\ref{b1})
can be transformed into an equation of form (\ref{b}) dividing the first
equation of (\ref{b1}) by its second one. So, the results on Eq. (\ref{b})
can be used to obtained the results on system (\ref{b1}). System (\ref{b1})
was studied for the first time in \cite{Grove0} by using the method
described above. Also, in \cite{Grove0}, it is shown that every positive
solution of system (\ref{b1}) is two-periodic or converges to its a
two-periodic solution. For more results on system (\ref{b1}), see \cite%
{Grove},\cite{Stevic-amc9223},\cite{S8}.

System (\ref{b1}) can be extended by interchanging the constant coefficients 
$a$, $b$, $c$, $d$ with two-periodic ones. More concretely, another
extension with two-periodic coefficients of (\ref{b1}) is the following
system of difference equations%
\begin{equation}
x_{n+1}=\frac{a_{n}}{x_{n}}+\frac{b_{n}}{y_{n}},\ y_{n+1}=\frac{c_{n}}{x_{n}}%
+\frac{d_{n}}{y_{n}},\ n\in \mathbb{N}_{0},  \label{sys}
\end{equation}%
where the initial values $x_{0},y_{0}$ are positive real numbers and $\left(
a_{n}\right) _{n\geq 0}$, $\left( b_{n}\right) _{n\geq 0}$, $\left(
c_{n}\right) _{n\geq 0}$, $\left( d_{n}\right) _{n\geq 0}$ are two-periodic
sequences of positive real numbers. For extensions with periodic
coefficients of some difference equations and systems, see \cite{Cinar},\cite%
{NT},\cite{Yalc}.

Our main purpose in this paper is to determine the long-term behavior of
posistive solutions of system (\ref{sys}). We also use a new method to solve
the system without needing some other nonlinear difference equations such as
(\ref{b}). Throughout this paper we assume that $a_{2n}=a_{0}$, $%
a_{2n+1}=a_{1}$, $b_{2n}=b_{0}$, $b_{2n+1}=b_{1}$, $c_{2n}=c_{0}$, $%
c_{2n+1}=c_{1}$ and $d_{2n}=d_{0}$, $d_{2n+1}=d_{1}$ with $a_{0}\neq a_{1}$, 
$b_{0}\neq b_{1}$, $c_{0}\neq c_{1}$ and $d_{0}\neq d_{1}$. We also adopt
the assumptions%
\begin{equation*}
\sum\limits_{k=m}^{m-l}s_{k}=0\text{ and }\prod\limits_{k=m}^{m-l}s_{k}=1,\
l\in \mathbb{N}.
\end{equation*}

\begin{definition}
A solution $\left( x_{n},y_{n}\right) _{n\geq 0}$ of the system%
\begin{equation*}
x_{n+1}=f\left( x_{n},y_{n}\right) ,\ y_{n+1}=g\left( x_{n},y_{n}\right) ,\
n\in \mathbb{N}_{0},
\end{equation*}
is eventually periodic with period $p$, if there is an $n_{1}>0$ such that $%
\left( x_{n+p},y_{n+p}\right) =\left( x_{n},y_{n}\right) $, then for $n\geq
n_{1}$. If $n_{1}=0$, then the solution is periodic with period $p$.
\end{definition}

The folowing result is extracted from \cite{Knopp}.

\begin{theorem}
\label{T}A product $\prod\limits_{k=0}^{\infty }\left( 1+\alpha _{k}\right) $
with positive terms $\alpha _{k}$ is convergent if and only if $%
\sum\limits_{k=0}^{\infty }\alpha _{k}$ converges.
\end{theorem}

\section{Main Results}

In this section we formulate and prove our main results.

\begin{theorem}
Assume that $x_{0},y_{0}>0$ and $\left( a_{n}\right) _{n\geq 0}$, $\left(
b_{n}\right) _{n\geq 0}$, $\left( c_{n}\right) _{n\geq 0}$, $\left(
d_{n}\right) _{n\geq 0}$ are two-periodic sequences of positive real
numbers. Then, system of difference equations (\ref{sys}) can be solved in
closed form.
\end{theorem}

\begin{proof}
First, it is easy to show by induction that $x_{n}$, $y_{n}>0$, for all $%
n\in N_{0}$. Multiplying both equations in (\ref{sys}) by the following
positive product%
\begin{equation*}
\prod\limits_{k=0}^{n}x_{k}y_{k},
\end{equation*}%
we obtain 
\begin{equation}
\prod\limits_{k=0}^{n+1}x_{k}\prod\limits_{k=0}^{n}y_{k}=a_{n}\prod%
\limits_{k=0}^{n-1}x_{k}\prod\limits_{k=0}^{n}y_{k}+b_{n}\prod%
\limits_{k=0}^{n}x_{k}\prod\limits_{k=0}^{n-1}y_{k}  \label{psx}
\end{equation}%
and%
\begin{equation}
\prod\limits_{k=0}^{n}x_{k}\prod\limits_{k=0}^{n+1}y_{k}=c_{n}\prod%
\limits_{k=0}^{n-1}x_{k}\prod\limits_{k=0}^{n}y_{k}+d_{n}\prod%
\limits_{k=0}^{n}x_{k}\prod\limits_{k=0}^{n-1}y_{k}  \label{psy}
\end{equation}%
for all $n\in \mathbb{N}_{0}$. Note that the equalities (\ref{psx})-(\ref%
{psy}) constitute a linear system with respect to the following products 
\begin{equation}
u_{n}=\prod\limits_{k=0}^{n}x_{k}\prod\limits_{k=0}^{n-1}y_{k}\text{ and }%
v_{n}=\prod\limits_{k=0}^{n-1}x_{k}\prod\limits_{k=0}^{n}y_{k}.  \label{cv}
\end{equation}%
Therefore, we can write this system in the vector form%
\begin{equation}
\left( 
\begin{array}{c}
u_{n+1} \\ 
v_{n+1}%
\end{array}%
\right) =\left( 
\begin{array}{cc}
b_{n} & a_{n} \\ 
d_{n} & c_{n}%
\end{array}%
\right) \left( 
\begin{array}{c}
u_{n} \\ 
v_{n}%
\end{array}%
\right) ,  \label{veksys}
\end{equation}%
where $u_{0}=x_{0}$, $v_{0}=y_{0}$, which is simplier, for all $n\in \mathbb{%
N}_{0}$. Let%
\begin{equation*}
A_{n}=\left( 
\begin{array}{cc}
b_{n} & a_{n} \\ 
d_{n} & c_{n}%
\end{array}%
\right) \text{.}
\end{equation*}%
Then, since the sequences $\left( a_{n}\right) _{n\geq 0}$, $\left(
b_{n}\right) _{n\geq 0}$, $\left( c_{n}\right) _{n\geq 0}$, $\left(
d_{n}\right) _{n\geq 0}$ are two-periodic, the matrix $A_{n}$ becomes%
\begin{equation*}
A_{n}=\left\{ 
\begin{array}{c}
A_{0}=\left( 
\begin{array}{cc}
b_{0} & a_{0} \\ 
d_{0} & c_{0}%
\end{array}%
\right) ,\ \text{if }n\text{ is even,} \\ 
A_{1}=\left( 
\begin{array}{cc}
b_{1} & a_{1} \\ 
d_{1} & c_{1}%
\end{array}%
\right) ,\ \text{if }n\text{ is odd.}%
\end{array}%
\right. 
\end{equation*}%
Now we decompose (\ref{veksys}) with respect to even-subscript and
odd-subscript terms as follows: 
\begin{eqnarray}
\left( 
\begin{array}{c}
u_{2n+1} \\ 
v_{2n+1}%
\end{array}%
\right)  &=&\left( 
\begin{array}{cc}
b_{0} & a_{0} \\ 
d_{0} & c_{0}%
\end{array}%
\right) \left( 
\begin{array}{c}
u_{2n} \\ 
v_{2n}%
\end{array}%
\right) ,  \label{odd} \\
\left( 
\begin{array}{c}
u_{2n+2} \\ 
v_{2n+2}%
\end{array}%
\right)  &=&\left( 
\begin{array}{cc}
b_{1} & a_{1} \\ 
d_{1} & c_{1}%
\end{array}%
\right) \left( 
\begin{array}{c}
u_{2n+1} \\ 
v_{2n+1}%
\end{array}%
\right)   \label{even}
\end{eqnarray}%
for all $n\in \mathbb{N}_{0}$. From which (\ref{odd}) and (\ref{even})
follows that%
\begin{equation*}
\left( 
\begin{array}{c}
u_{2n+2} \\ 
v_{2n+2}%
\end{array}%
\right) =\left( 
\begin{array}{cc}
b_{1} & a_{1} \\ 
d_{1} & c_{1}%
\end{array}%
\right) \left( 
\begin{array}{cc}
b_{0} & a_{0} \\ 
d_{0} & c_{0}%
\end{array}%
\right) \left( 
\begin{array}{c}
u_{2n} \\ 
v_{2n}%
\end{array}%
\right) 
\end{equation*}%
or%
\begin{equation}
\left( 
\begin{array}{c}
u_{2n+2} \\ 
v_{2n+2}%
\end{array}%
\right) =\left( 
\begin{array}{cc}
a_{1}d_{0}+b_{0}b_{1} & a_{0}b_{1}+a_{1}c_{0} \\ 
b_{0}d_{1}+c_{1}d_{0} & a_{0}d_{1}+c_{0}c_{1}%
\end{array}%
\right) \left( 
\begin{array}{c}
u_{2n} \\ 
v_{2n}%
\end{array}%
\right) .  \label{sysu2n1}
\end{equation}%
Let $A_{0}A_{1}=A$. Then, we consider two cases of the matrix $A$ as the
following:

Case 1: $rank\left( A\right) =1$. In this case the first row in the matrix $%
A $ are linearly dependent on the second one. Without loss of generality we
may assume that%
\begin{equation}
\left( b_{0}d_{1}+c_{1}d_{0},a_{0}d_{1}+c_{0}c_{1}\right) =K\left(
a_{1}d_{0}+b_{0}b_{1},a_{0}b_{1}+a_{1}c_{0}\right) ,  \label{LB}
\end{equation}%
where $K$ is a positive constant such that%
\begin{equation*}
K=\frac{b_{0}d_{1}+c_{1}d_{0}}{a_{1}d_{0}+b_{0}b_{1}}=\frac{%
a_{0}d_{1}+c_{0}c_{1}}{a_{0}b_{1}+a_{1}c_{0}}.
\end{equation*}%
Using (\ref{LB}) in system (\ref{sysu2n1}) we have%
\begin{eqnarray*}
u_{2n+2} &=&\left( a_{1}d_{0}+b_{0}b_{1}\right) u_{2n}+\left(
a_{0}b_{1}+a_{1}c_{0}\right) v_{2n}, \\
v_{2n+2} &=&K\left( \left( a_{1}d_{0}+b_{0}b_{1}\right) u_{2n}+\left(
a_{0}b_{1}+a_{1}c_{0}\right) v_{2n}\right)
\end{eqnarray*}%
which implies the relation 
\begin{equation*}
v_{2n+2}=Ku_{2n+2}
\end{equation*}%
for all $n\in \mathbb{N}_{0}$. By the last three relations we have 
\begin{equation*}
u_{2n+2}=\left( a_{1}d_{0}+b_{0}b_{1}+K\left( a_{0}b_{1}+a_{1}c_{0}\right)
\right) u_{2n}
\end{equation*}%
from which it follows that%
\begin{eqnarray}
u_{2n} &=&\left( a_{1}d_{0}+b_{0}b_{1}+K\left( a_{0}b_{1}+a_{1}c_{0}\right)
\right) ^{n-1}u_{2},  \label{u0} \\
v_{2n} &=&K\left( a_{1}d_{0}+b_{0}b_{1}+K\left( a_{0}b_{1}+a_{1}c_{0}\right)
\right) ^{n-1}u_{2}  \label{v0}
\end{eqnarray}%
for all $n\in \mathbb{N}$. Using (\ref{u0}) and (\ref{v0}) in (\ref{odd}) we
obtain 
\begin{eqnarray}
u_{2n+1} &=&\left( b_{0}+Ka_{0}\right) \left( a_{1}d_{0}+b_{0}b_{1}+K\left(
a_{0}b_{1}+a_{1}c_{0}\right) \right) ^{n-1}u_{2}  \label{u1} \\
v_{2n+1} &=&\left( d_{0}+Kc_{0}\right) \left( a_{1}d_{0}+b_{0}b_{1}+K\left(
a_{0}b_{1}+a_{1}c_{0}\right) \right) ^{n-1}u_{2}  \label{v1}
\end{eqnarray}%
for all $n\in \mathbb{N}$. Also, the changes of variables in (\ref{cv}) yield%
\begin{eqnarray}
x_{n+1} &=&\frac{u_{n+1}v_{n-1}}{u_{n}v_{n}}x_{n-1},  \label{xuv} \\
y_{n+1} &=&\frac{v_{n+1}u_{n-1}}{v_{n}u_{n}}x_{n-1}  \label{yuv}
\end{eqnarray}%
for all $n\in \mathbb{N}$. Hence, from (\ref{xuv}) and (\ref{yuv}), we obtain%
\begin{eqnarray}
x_{2n} &=&x_{0}\prod\limits_{k=1}^{n}\frac{u_{2k}v_{2k-2}}{u_{2k-1}v_{2k-1}}
\label{h1} \\
x_{2n+1} &=&x_{1}\prod\limits_{k=1}^{n}\frac{u_{2k+1}v_{2k-1}}{u_{2k}v_{2k}}
\label{h2}
\end{eqnarray}%
and%
\begin{eqnarray}
y_{2n} &=&y_{0}\prod\limits_{k=0}^{n}\frac{v_{2k}u_{2k-2}}{v_{2k-1}u_{2k-1}},
\label{h3} \\
y_{2n+1} &=&y_{1}\prod\limits_{k=1}^{n}\frac{v_{2k+1}u_{2k-1}}{v_{2k}u_{2k}},
\label{h4}
\end{eqnarray}%
respectively. By employing (\ref{odd}) in (\ref{h1})-(\ref{h4}), we have the
following closed formulas 
\begin{eqnarray}
x_{2n} &=&x_{0}\prod\limits_{k=1}^{n}\frac{u_{2k}v_{2k-2}}{\left(
b_{0}u_{2k-2}+a_{0}v_{2k-2}\right) \left( d_{0}u_{2k-2}+c_{0}v_{2k-2}\right) 
},  \label{px2n} \\
x_{2n+1} &=&x_{1}\prod\limits_{k=1}^{n}\frac{\left(
b_{0}u_{2k}+a_{0}v_{2k}\right) \left( d_{0}u_{2k-2}+c_{0}v_{2k-2}\right) }{%
u_{2k}v_{2k}},  \label{px2n+1} \\
y_{2n} &=&y_{0}\prod\limits_{k=1}^{n}\frac{v_{2k}u_{2k-2}}{\left(
d_{0}u_{2k-2}+c_{0}v_{2k-2}\right) \left( b_{0}u_{2k-2}+a_{0}v_{2k-2}\right) 
},  \label{py2n} \\
y_{2n+1} &=&y_{1}\prod\limits_{k=1}^{n}\frac{\left(
d_{0}u_{2k}+c_{0}v_{2k}\right) \left( b_{0}u_{2k-2}+a_{0}v_{2k-2}\right) }{%
v_{2k}u_{2k}},  \label{py2n+1}
\end{eqnarray}%
which is valid for all $n\in \mathbb{N}_{0}$, respectively. Consequently, in
the case $rank\left( A\right) =1$, by using the formulas (\ref{u0})-(\ref{v1}%
) in (\ref{px2n})-(\ref{py2n+1}), we have the general solution of (\ref{sys}%
) as follows: 
\begin{eqnarray}
x_{2n} &=&x_{0}\prod\limits_{k=1}^{n}\frac{K\left(
a_{1}d_{0}+b_{0}b_{1}+K\left( a_{0}b_{1}+a_{1}c_{0}\right) \right) }{\left(
b_{0}+Ka_{0}\right) \left( d_{0}+Kc_{0}\right) }=x_{0}\left( \frac{K\left(
a_{1}d_{0}+b_{0}b_{1}+K\left( a_{0}b_{1}+a_{1}c_{0}\right) \right) }{\left(
b_{0}+Ka_{0}\right) \left( d_{0}+Kc_{0}\right) }\right) ^{n},  \label{1x2n}
\\
x_{2n+1} &=&x_{1}\prod\limits_{k=1}^{n}\frac{\left( b_{0}+Ka_{0}\right)
\left( d_{0}+Kc_{0}\right) }{K\left( a_{1}d_{0}+b_{0}b_{1}+K\left(
a_{0}b_{1}+a_{1}c_{0}\right) \right) }=x_{1}\left( \frac{\left(
b_{0}+Ka_{0}\right) \left( d_{0}+Kc_{0}\right) }{K\left(
a_{1}d_{0}+b_{0}b_{1}+K\left( a_{0}b_{1}+a_{1}c_{0}\right) \right) }\right)
^{n},  \label{1x2n+1} \\
y_{2n} &=&y_{0}\prod\limits_{k=1}^{n}\frac{K\left(
a_{1}d_{0}+b_{0}b_{1}+K\left( a_{0}b_{1}+a_{1}c_{0}\right) \right) }{\left(
d_{0}+Kc_{0}\right) \left( b_{0}+Ka_{0}\right) }=y_{0}\left( \frac{K\left(
a_{1}d_{0}+b_{0}b_{1}+K\left( a_{0}b_{1}+a_{1}c_{0}\right) \right) }{\left(
d_{0}+Kc_{0}\right) \left( b_{0}+Ka_{0}\right) }\right) ^{n},  \label{1y2n}
\\
y_{2n+1} &=&y_{1}\prod\limits_{k=1}^{n}\frac{\left( d_{0}+Kc_{0}\right)
\left( b_{0}+Ka_{0}\right) }{K\left( a_{1}d_{0}+b_{0}b_{1}+K\left(
a_{0}b_{1}+a_{1}c_{0}\right) \right) }=y_{1}\left( \frac{\left(
d_{0}+Kc_{0}\right) \left( b_{0}+Ka_{0}\right) }{K\left(
a_{1}d_{0}+b_{0}b_{1}+K\left( a_{0}b_{1}+a_{1}c_{0}\right) \right) }\right)
^{n},  \label{1y2n+1}
\end{eqnarray}%
where%
\begin{equation}
K=\frac{b_{0}d_{1}+c_{1}d_{0}}{a_{1}d_{0}+b_{0}b_{1}}=\frac{%
a_{0}d_{1}+c_{0}c_{1}}{a_{0}b_{1}+a_{1}c_{0}}.  \label{k}
\end{equation}%
for all $n\in \mathbb{N}_{0}$, respectively.

Case 2: $rank\left( A\right) =2$. In this case both rows in the matrix $A$
are linearly independent of each other. This case also implies that $A$ has
two different eigenvalues given by%
\begin{equation}
\lambda _{1}=\frac{a_{1}d_{0}+b_{0}b_{1}+a_{0}d_{1}+c_{0}c_{1}+\sqrt{\left(
a_{1}d_{0}+b_{0}b_{1}-a_{0}d_{1}-c_{0}c_{1}\right) ^{2}+4\left(
a_{0}b_{1}+a_{1}c_{0}\right) \left( b_{0}d_{1}+c_{1}d_{0}\right) }}{2}
\label{r1}
\end{equation}%
and 
\begin{equation}
\lambda _{2}=\frac{a_{1}d_{0}+b_{0}b_{1}+a_{0}d_{1}+c_{0}c_{1}-\sqrt{\left(
a_{1}d_{0}+b_{0}b_{1}-a_{0}d_{1}-c_{0}c_{1}\right) ^{2}+4\left(
a_{0}b_{1}+a_{1}c_{0}\right) \left( b_{0}d_{1}+c_{1}d_{0}\right) }}{2}.
\label{r2}
\end{equation}

Since these eigenvalues will correspond to two linear independent
eigenvectors, we may write the matrix $A$ as follows:%
\begin{equation*}
A=P\Lambda P^{-1}
\end{equation*}%
where%
\begin{equation*}
P=\left( 
\begin{array}{cc}
\frac{a_{0}b_{1}+a_{1}c_{0}}{\lambda _{1}-\left(
a_{1}d_{0}+b_{0}b_{1}\right) } & \frac{a_{0}b_{1}+a_{1}c_{0}}{\lambda
_{2}-\left( a_{1}d_{0}+b_{0}b_{1}\right) } \\ 
1 & 1%
\end{array}%
\right) ,
\end{equation*}%
\begin{equation*}
P^{-1}=\left( 
\begin{array}{cc}
\frac{b_{0}d_{1}+c_{1}d_{0}}{\lambda _{1}-\lambda _{2}} & \frac{\lambda
_{1}-\left( a_{1}d_{0}+b_{0}b_{1}\right) }{\lambda _{1}-\lambda _{2}} \\ 
-\frac{b_{0}d_{1}+c_{1}d_{0}}{\lambda _{1}-\lambda _{2}} & -\frac{\lambda
_{2}-\left( a_{1}d_{0}+b_{0}b_{1}\right) }{\lambda _{1}-\lambda _{2}}%
\end{array}%
\right)
\end{equation*}%
and 
\begin{equation*}
\Lambda =\left( 
\begin{array}{cc}
\lambda _{1} & 0 \\ 
0 & \lambda _{2}%
\end{array}%
\right) .
\end{equation*}%
Therefore we may write system (\ref{sysu2n1})\ as the following 
\begin{equation}
Z_{2n+2}=P\Lambda P^{-1}Z_{2n},  \label{Z2n+2}
\end{equation}%
where%
\begin{equation*}
Z_{2n}=\left( 
\begin{array}{c}
u_{2n} \\ 
v_{2n}%
\end{array}%
\right) ,
\end{equation*}%
for all $n\in \mathbb{N}_{0}$. From (\ref{Z2n+2}) we have%
\begin{equation*}
P^{-1}Z_{2n+2}=\Lambda P^{-1}Z_{2n}
\end{equation*}%
from which it follows that%
\begin{equation}
P^{-1}Z_{2n}=\Lambda ^{n}P^{-1}Z_{0}  \label{Z2n}
\end{equation}%
for all $n\in \mathbb{N}_{0}$. Multiplying both sides of (\ref{Z2n}) by the
matrix $P$, we have%
\begin{equation}
Z_{2n}=P\Lambda ^{n}P^{-1}Z_{0}  \label{Z2na}
\end{equation}%
or after some computations%
\begin{equation*}
\left( 
\begin{array}{c}
u_{2n} \\ 
v_{2n}%
\end{array}%
\right) =\left( 
\begin{array}{c}
C_{1}\lambda _{1}^{n}-C_{2}\lambda _{2}^{n} \\ 
C_{3}\lambda _{1}^{n}-C_{4}\lambda _{2}^{n}%
\end{array}%
\right) ,
\end{equation*}%
where%
\begin{eqnarray}
C_{1} &=&\frac{a_{0}b_{1}+a_{1}c_{0}}{\lambda _{1}-\lambda _{2}}\left( \frac{%
b_{0}d_{1}+c_{1}d_{0}}{\lambda _{1}-\left( a_{1}d_{0}+b_{0}b_{1}\right) }%
u_{0}+v_{0}\right) ,  \label{c1} \\
C_{2} &=&\frac{a_{0}b_{1}+a_{1}c_{0}}{\lambda _{1}-\lambda _{2}}\left( \frac{%
b_{0}d_{1}+c_{1}d_{0}}{\lambda _{2}-\left( a_{1}d_{0}+b_{0}b_{1}\right) }%
u_{0}+v_{0}\right) ,  \label{c2} \\
C_{3} &=&\frac{b_{0}d_{1}+c_{1}d_{0}}{\lambda _{1}-\lambda _{2}}u_{0}+\frac{%
\lambda _{1}-\left( a_{1}d_{0}+b_{0}b_{1}\right) }{\lambda _{1}-\lambda _{2}}%
v_{0},  \label{c3} \\
C_{4} &=&\frac{b_{0}d_{1}+c_{1}d_{0}}{\lambda _{1}-\lambda _{2}}u_{0}+\frac{%
\lambda _{2}-\left( a_{1}d_{0}+b_{0}b_{1}\right) }{\lambda _{1}-\lambda _{2}}%
v_{0},  \label{c4}
\end{eqnarray}%
for all $n\in \mathbb{N}_{0}$. From the last vectorial equality, we have%
\begin{eqnarray}
u_{2n} &=&C_{1}\lambda _{1}^{n}-C_{2}\lambda _{2}^{n}  \label{u2n} \\
v_{2n} &=&C_{3}\lambda _{1}^{n}-C_{4}\lambda _{2}^{n}  \label{v2n}
\end{eqnarray}%
for all $n\in \mathbb{N}_{0}$. From (\ref{odd}), (\ref{u2n}) and (\ref{v2n}%
), we have the formulas%
\begin{equation}
u_{2n+1}=\left( b_{0}C_{1}+a_{0}C_{3}\right) \lambda _{1}^{n}-\left(
b_{0}C_{2}+a_{0}C_{4}\right) \lambda _{2}^{n}  \label{u2n+1}
\end{equation}%
and%
\begin{equation}
v_{2n+1}=\left( d_{0}C_{1}+c_{0}C_{3}\right) \lambda _{1}^{n}-\left(
d_{0}C_{2}+c_{0}C_{4}\right) \lambda _{2}^{n}  \label{v2n+1}
\end{equation}%
for all $n\in \mathbb{N}_{0}$. Also, we can write the formulas (\ref{px2n})-(%
\ref{py2n+1}) as the following:%
\begin{eqnarray}
x_{2n} &=&x_{0}\prod\limits_{k=1}^{n}\frac{\frac{u_{2k}}{v_{2k-2}}}{\left(
b_{0}\frac{u_{2k-2}}{v_{2k-2}}+a_{0}\right) \left( d_{0}\frac{u_{2k-2}}{%
v_{2k-2}}+c_{0}\right) },  \label{x2n} \\
x_{2n+1} &=&x_{1}\prod\limits_{k=1}^{n}\frac{\left( b_{0}\frac{u_{2k}}{v_{2k}%
}+a_{0}\right) \left( d_{0}\frac{u_{2k-2}}{v_{2k-2}}+c_{0}\right) }{\frac{%
u_{2k}}{v_{2k-2}}},  \label{x2n+1} \\
y_{2n} &=&y_{0}\prod\limits_{k=1}^{n}\frac{\frac{v_{2k}}{u_{2k-2}}}{\left(
d_{0}+c_{0}\frac{v_{2k-2}}{u_{2k-2}}\right) \left( b_{0}+a_{0}\frac{v_{2k-2}%
}{u_{2k-2}}\right) },  \label{y2n} \\
y_{2n+1} &=&y_{1}\prod\limits_{k=1}^{n}\frac{\left( d_{0}+c_{0}\frac{v_{2k}}{%
u_{2k}}\right) \left( b_{0}+a_{0}\frac{v_{2k-2}}{u_{2k-2}}\right) }{\frac{%
v_{2k}}{u_{2k-2}}}  \label{y2n+1}
\end{eqnarray}%
for all $n\in \mathbb{N}_{0}$. Finally, by employing (\ref{u2n})-(\ref{v2n+1}%
) in (\ref{x2n})-(\ref{y2n+1}), we have the general solution of (\ref{sys})
as the following%
\begin{eqnarray*}
x_{2n} &=&x_{0}\prod\limits_{k=1}^{n}\frac{\frac{C_{1}\lambda
_{1}^{k}-C_{2}\lambda _{2}^{k}}{C_{3}\lambda _{1}^{k-1}-C_{4}\lambda
_{2}^{k-1}}}{\left( b_{0}\frac{C_{1}\lambda _{1}^{k-1}-C_{2}\lambda
_{2}^{k-1}}{C_{3}\lambda _{1}^{k-1}-C_{4}\lambda _{2}^{k-1}}+a_{0}\right)
\left( d_{0}\frac{C_{1}\lambda _{1}^{k-1}-C_{2}\lambda _{2}^{k-1}}{%
C_{3}\lambda _{1}^{k-1}-C_{4}\lambda _{2}^{k-1}}+c_{0}\right) } \\
x_{2n+1} &=&x_{1}\prod\limits_{k=1}^{n}\frac{\left( b_{0}\frac{C_{1}\lambda
_{1}^{k}-C_{2}\lambda _{2}^{k}}{C_{3}\lambda _{1}^{k}-C_{4}\lambda _{2}^{k}}%
+a_{0}\right) \left( d_{0}\frac{C_{1}\lambda _{1}^{k-1}-C_{2}\lambda
_{2}^{k-1}}{C_{3}\lambda _{1}^{k-1}-C_{4}\lambda _{2}^{k-1}}+c_{0}\right) }{%
\frac{C_{1}\lambda _{1}^{k}-C_{2}\lambda _{2}^{k}}{C_{3}\lambda
_{1}^{k-1}-C_{4}\lambda _{2}^{k-1}}} \\
y_{2n} &=&y_{0}\prod\limits_{k=1}^{n}\frac{\frac{C_{3}\lambda
_{1}^{k}-C_{4}\lambda _{2}^{k}}{C_{1}\lambda _{1}^{k-1}-C_{2}\lambda
_{2}^{k-1}}}{\left( d_{0}+c_{0}\frac{C_{3}\lambda _{1}^{k-1}-C_{4}\lambda
_{2}^{k-1}}{C_{1}\lambda _{1}^{k-1}-C_{2}\lambda _{2}^{k-1}}\right) \left(
b_{0}+a_{0}\frac{C_{3}\lambda _{1}^{k-1}-C_{4}\lambda _{2}^{k-1}}{%
C_{1}\lambda _{1}^{k-1}-C_{2}\lambda _{2}^{k-1}}\right) } \\
y_{2n+1} &=&y_{1}\prod\limits_{k=1}^{n}\frac{\left( d_{0}+c_{0}\frac{%
C_{3}\lambda _{1}^{k}-C_{4}\lambda _{2}^{k}}{C_{1}\lambda
_{1}^{k}-C_{2}\lambda _{2}^{k}}\right) \left( b_{0}+a_{0}\frac{C_{3}\lambda
_{1}^{k-1}-C_{4}\lambda _{2}^{k-1}}{C_{1}\lambda _{1}^{k-1}-C_{2}\lambda
_{2}^{k-1}}\right) }{\frac{C_{3}\lambda _{1}^{k}-C_{4}\lambda _{2}^{k}}{%
C_{1}\lambda _{1}^{k-1}-C_{2}\lambda _{2}^{k-1}}}
\end{eqnarray*}%
for all $n\in \mathbb{N}_{0}$.
\end{proof}

The following theorem determines and characterizes the long-term behavior of
positive solutions of (\ref{sys}) according to the parameters in the case $%
rank\left( A\right) =1$.

\begin{theorem}
Assume that $x_{0}$, $y_{0}>0$ and $\left( a_{n}\right) _{n\geq 0}$, $\left(
b_{n}\right) _{n\geq 0}$, $\left( c_{n}\right) _{n\geq 0}$, $\left(
d_{n}\right) _{n\geq 0}$ are two-periodic sequences of positive real
numbers. If%
\begin{equation*}
rank\left( 
\begin{array}{cc}
a_{1}d_{0}+b_{0}b_{1} & a_{0}b_{1}+a_{1}c_{0} \\ 
b_{0}d_{1}+c_{1}d_{0} & a_{0}d_{1}+c_{0}c_{1}%
\end{array}%
\right) =1,
\end{equation*}%
then, for the solutions of system (\ref{sys}) the followings are true:

(i) If $\frac{K\left( a_{1}d_{0}+b_{0}b_{1}+K\left(
a_{0}b_{1}+a_{1}c_{0}\right) \right) }{\left( b_{0}+Ka_{0}\right) \left(
d_{0}+Kc_{0}\right) }<1$, then $x_{2n}\rightarrow 0$, $x_{2n+1}\rightarrow
\infty $, $y_{2n}\rightarrow 0$, $x_{2n+1}\rightarrow \infty $ as $%
n\rightarrow \infty $.

(ii) If $\frac{K\left( a_{1}d_{0}+b_{0}b_{1}+K\left(
a_{0}b_{1}+a_{1}c_{0}\right) \right) }{\left( b_{0}+Ka_{0}\right) \left(
d_{0}+Kc_{0}\right) }>1$, then $x_{2n}\rightarrow \infty $, $%
x_{2n+1}\rightarrow 0$, $y_{2n}\rightarrow \infty $, $x_{2n+1}\rightarrow 0$
as $n\rightarrow \infty $.

(iii) If $\frac{K\left( a_{1}d_{0}+b_{0}b_{1}+K\left(
a_{0}b_{1}+a_{1}c_{0}\right) \right) }{\left( b_{0}+Ka_{0}\right) \left(
d_{0}+Kc_{0}\right) }=1$, then every solution of (\ref{sys}) is
two-periodic, where $K$ is given by (\ref{k}).
\end{theorem}

\begin{proof}
The proof follows directly from formulas (\ref{1x2n})-(\ref{1y2n+1}). That
is to say, it is clearly seen from these formulas that if%
\begin{equation*}
\frac{K\left( a_{1}d_{0}+b_{0}b_{1}+K\left( a_{0}b_{1}+a_{1}c_{0}\right)
\right) }{\left( b_{0}+Ka_{0}\right) \left( d_{0}+Kc_{0}\right) }<1,
\end{equation*}%
then $x_{2n}\rightarrow 0$, $x_{2n+1}\rightarrow \infty $, $%
y_{2n}\rightarrow 0$, $x_{2n+1}\rightarrow \infty $ as $n\rightarrow \infty $%
. If%
\begin{equation*}
\frac{K\left( a_{1}d_{0}+b_{0}b_{1}+K\left( a_{0}b_{1}+a_{1}c_{0}\right)
\right) }{\left( b_{0}+Ka_{0}\right) \left( d_{0}+Kc_{0}\right) }>1,
\end{equation*}%
then $x_{2n}\rightarrow \infty $, $x_{2n+1}\rightarrow 0$, $%
y_{2n}\rightarrow \infty $, $x_{2n+1}\rightarrow 0$ as $n\rightarrow \infty $%
. If%
\begin{equation*}
\frac{K\left( a_{1}d_{0}+b_{0}b_{1}+K\left( a_{0}b_{1}+a_{1}c_{0}\right)
\right) }{\left( b_{0}+Ka_{0}\right) \left( d_{0}+Kc_{0}\right) }=1,
\end{equation*}%
then every solution of (\ref{sys}) is two-periodic such that $x_{2n}=x_{0}$, 
$x_{2n+1}=x_{1}$, $y_{2n}=y_{0}$, $y_{2n+1}=y_{1}$.
\end{proof}

The following theorem determines and characterizes the long-term behavior of
positive solutions of (\ref{sys}) according to the parameters in the case $%
rank\left( A\right) =2$.

\begin{theorem}
Assume that $x_{0}$, $y_{0}>0$ and $\left( a_{n}\right) _{n\geq 0}$, $\left(
b_{n}\right) _{n\geq 0}$, $\left( c_{n}\right) _{n\geq 0}$, $\left(
d_{n}\right) _{n\geq 0}$ are two-periodic sequences of positive real
numbers. If%
\begin{equation*}
rank\left( 
\begin{array}{cc}
a_{1}d_{0}+b_{0}b_{1} & a_{0}b_{1}+a_{1}c_{0} \\ 
b_{0}d_{1}+c_{1}d_{0} & a_{0}d_{1}+c_{0}c_{1}%
\end{array}%
\right) =2,
\end{equation*}%
then, for the solutions of system (\ref{sys}) the followings are true:

(i) If $\frac{C_{1}\lambda _{1}}{C_{3}}-\left( b_{0}\frac{C_{1}}{C_{3}}%
+a_{0}\right) \left( d_{0}\frac{C_{1}}{C_{3}}+c_{0}\right) <0$, then $%
x_{2n}\rightarrow 0$, $x_{2n+1}\rightarrow \infty $, $y_{2n}\rightarrow 0$, $%
x_{2n+1}\rightarrow \infty $ as $n\rightarrow \infty $.

(ii) If $\frac{C_{1}\lambda _{1}}{C_{3}}-\left( b_{0}\frac{C_{1}}{C_{3}}%
+a_{0}\right) \left( d_{0}\frac{C_{1}}{C_{3}}+c_{0}\right) >0$, then $%
x_{2n}\rightarrow \infty $, $x_{2n+1}\rightarrow 0$, $y_{2n}\rightarrow
\infty $, $x_{2n+1}\rightarrow 0$ as $n\rightarrow \infty $.

(iii) If $\frac{C_{1}\lambda _{1}}{C_{3}}-\left( b_{0}\frac{C_{1}}{C_{3}}%
+a_{0}\right) \left( d_{0}\frac{C_{1}}{C_{3}}+c_{0}\right) =0$, then every
solution of (\ref{sys}) converges to a two-periodic positive solution of the
system, where $\lambda _{1}$ is given by (\ref{r1}) and $C_{1}$, $C_{3}$ are
given by (\ref{c1}) and (\ref{c3})
\end{theorem}

\begin{proof}
(i)-(ii) It is obvious from (\ref{r1})~and (\ref{r2}) that $\left\vert
\lambda _{2}\right\vert <\left\vert \lambda _{1}\right\vert $, since the
sequences $\left( a_{n}\right) _{n\geq 0}$, $\left( b_{n}\right) _{n\geq 0}$%
, $\left( c_{n}\right) _{n\geq 0}$, $\left( d_{n}\right) _{n\geq 0}$ are
positive. Let%
\begin{eqnarray*}
p_{k} &=&\frac{u_{2k}}{v_{2k-2}}=\frac{C_{1}\lambda _{1}^{k}-C_{2}\lambda
_{2}^{k}}{C_{3}\lambda _{1}^{k-1}-C_{4}\lambda _{2}^{k-1}}, \\
q_{k} &=&\frac{u_{2k}}{v_{2k}}=\frac{C_{1}\lambda _{1}^{k}-C_{2}\lambda
_{2}^{k}}{C_{3}\lambda _{1}^{k}-C_{4}\lambda _{2}^{k}}, \\
r_{k} &=&\frac{v_{2k}}{u_{2k-2}}=\frac{C_{3}\lambda _{1}^{k}-C_{4}\lambda
_{2}^{k}}{C_{1}\lambda _{1}^{k-1}-C_{2}\lambda _{2}^{k-1}}, \\
s_{k} &=&\frac{v_{2k}}{u_{2k}}=\frac{C_{3}\lambda _{1}^{k}-C_{4}\lambda
_{2}^{k}}{C_{1}\lambda _{1}^{k}-C_{2}\lambda _{2}^{k}}.
\end{eqnarray*}%
Then, we have 
\begin{eqnarray}
x_{2n} &=&x_{0}\prod\limits_{k=1}^{n}\frac{p_{k}}{\left(
b_{0}q_{k-1}+a_{0}\right) \left( d_{0}q_{k-1}+c_{0}\right) }%
=x_{0}\prod\limits_{k=1}^{n}\left( 1+\frac{p_{k}-\left(
b_{0}q_{k-1}+a_{0}\right) \left( d_{0}q_{k-1}+c_{0}\right) }{\left(
b_{0}q_{k-1}+a_{0}\right) \left( d_{0}q_{k-1}+c_{0}\right) }\right) ,
\label{t1} \\
x_{2n+1} &=&x_{1}\prod\limits_{k=1}^{n}\frac{\left( b_{0}q_{k}+a_{0}\right)
\left( d_{0}q_{k-1}+c_{0}\right) }{p_{k}}=\frac{x_{1}}{\prod%
\limits_{k=1}^{n}\left( 1+\frac{p_{k}-\left( b_{0}q_{k}+a_{0}\right) \left(
d_{0}q_{k-1}+c_{0}\right) }{\left( b_{0}q_{k}+a_{0}\right) \left(
d_{0}q_{k-1}+c_{0}\right) }\right) },  \label{t2} \\
y_{2n} &=&y_{0}\prod\limits_{k=1}^{n}\frac{r_{k}}{\left(
d_{0}+c_{0}s_{k-1}\right) \left( b_{0}+a_{0}s_{k-1}\right) }%
=y_{0}\prod\limits_{k=1}^{n}\left( 1+\frac{r_{k}-\left(
d_{0}+c_{0}s_{k-1}\right) \left( b_{0}+a_{0}s_{k-1}\right) }{\left(
d_{0}+c_{0}s_{k-1}\right) \left( b_{0}+a_{0}s_{k-1}\right) }\right) ,
\label{t3} \\
y_{2n+1} &=&y_{1}\prod\limits_{k=1}^{n}\frac{\left( d_{0}+c_{0}s_{k}\right)
\left( b_{0}+a_{0}s_{k-1}\right) }{r_{k}}=\frac{y_{1}}{\prod%
\limits_{k=1}^{n}\left( 1+\frac{r_{k}-\left( d_{0}+c_{0}s_{k}\right) \left(
b_{0}+a_{0}s_{k-1}\right) }{\left( d_{0}+c_{0}s_{k}\right) \left(
b_{0}+a_{0}s_{k-1}\right) }\right) }.  \label{t4}
\end{eqnarray}%
Since%
\begin{equation}
\lim_{k\rightarrow \infty }p_{k}=\frac{C_{1}\lambda _{1}}{C_{3}},\text{ }%
\lim_{k\rightarrow \infty }q_{k}=\frac{C_{1}}{C_{3}},\ \lim_{k\rightarrow
\infty }r_{k}=\frac{C_{3}\lambda _{1}}{C_{1}},\text{ }\lim_{k\rightarrow
\infty }s_{k}=\frac{C_{3}}{C_{1}},  \label{L}
\end{equation}%
we have the limits%
\begin{eqnarray*}
L_{1} &=&\lim_{k\rightarrow \infty }\frac{p_{k}-\left(
b_{0}q_{k-1}+a_{0}\right) \left( d_{0}q_{k-1}+c_{0}\right) }{\left(
b_{0}q_{k-1}+a_{0}\right) \left( d_{0}q_{k-1}+c_{0}\right) }=\frac{\frac{%
C_{1}\lambda _{1}}{C_{3}}-\left( b_{0}\frac{C_{1}}{C_{3}}+a_{0}\right)
\left( d_{0}\frac{C_{1}}{C_{3}}+c_{0}\right) }{\left( b_{0}\frac{C_{1}}{C_{3}%
}+a_{0}\right) \left( d_{0}\frac{C_{1}}{C_{3}}+c_{0}\right) }, \\
L_{2} &=&\lim_{k\rightarrow \infty }\frac{r_{k}-\left(
d_{0}+c_{0}s_{k-1}\right) \left( b_{0}+a_{0}s_{k-1}\right) }{\left(
d_{0}+c_{0}s_{k-1}\right) \left( b_{0}+a_{0}s_{k-1}\right) }=\frac{\frac{%
C_{3}\lambda _{1}}{C_{1}}-\left( b_{0}+a_{0}\frac{C_{3}}{C_{1}}\right)
\left( d_{0}+c_{0}\frac{C_{3}}{C_{1}}\right) }{\left( b_{0}+a_{0}\frac{C_{3}%
}{C_{1}}\right) \left( d_{0}+c_{0}\frac{C_{3}}{C_{1}}\right) },
\end{eqnarray*}%
where 
\begin{equation*}
\frac{C_{1}}{C_{3}}=\frac{a_{0}b_{1}+a_{1}c_{0}}{\lambda _{1}-\left(
a_{1}d_{0}+b_{0}b_{1}\right) }.
\end{equation*}%
Note that $L_{1}=L_{2}$. Hereby, convergence characters of the infinite
series%
\begin{equation}
\sum\limits_{k=0}^{\infty }\frac{p_{k}-\left( b_{0}q_{k-1}+a_{0}\right)
\left( d_{0}q_{k-1}+c_{0}\right) }{\left( b_{0}q_{k-1}+a_{0}\right) \left(
d_{0}q_{k-1}+c_{0}\right) }  \label{s1}
\end{equation}%
and%
\begin{equation}
\sum\limits_{k=0}^{\infty }\frac{r_{k}-\left( d_{0}+c_{0}s_{k-1}\right)
\left( b_{0}+a_{0}s_{k-1}\right) }{\left( d_{0}+c_{0}s_{k-1}\right) \left(
b_{0}+a_{0}s_{k-1}\right) }  \label{s2}
\end{equation}%
are same. We can say from a well-known fundamental result about infinite
series that (\ref{s1}) and (\ref{s2}) are divergent, if $\frac{C_{1}\lambda
_{1}}{C_{3}}-\left( b_{0}\frac{C_{1}}{C_{3}}+a_{0}\right) \left( d_{0}\frac{%
C_{1}}{C_{3}}+c_{0}\right) \neq 0$. So, the proofs of items (i)-(ii) follow
from these considerations and Theorem \ref{T}.

(iii) From (\ref{s1}), for sufficiently large $n_{0}$, we have%
\begin{equation*}
\sum\limits_{k=0}^{\infty }\frac{p_{k}-\left( b_{0}q_{k-1}+a_{0}\right)
\left( d_{0}q_{k-1}+c_{0}\right) }{\left( b_{0}q_{k-1}+a_{0}\right) \left(
d_{0}q_{k-1}+c_{0}\right) }=S_{1}\left( n_{0}\right) +K_{1}\left(
n_{0}\right)
\end{equation*}%
where%
\begin{equation*}
S_{1}\left( n_{0}\right) =\sum\limits_{k=0}^{n_{0}}\frac{p_{k}-\left(
b_{0}q_{k-1}+a_{0}\right) \left( d_{0}q_{k-1}+c_{0}\right) }{\left(
b_{0}q_{k-1}+a_{0}\right) \left( d_{0}q_{k-1}+c_{0}\right) }\text{ and }%
K_{1}\left( n_{0}\right) =\sum\limits_{k=n_{0}}^{\infty }\frac{\frac{%
C_{1}\lambda _{1}}{C_{3}}-\left( b_{0}\frac{C_{1}}{C_{3}}+a_{0}\right)
\left( d_{0}\frac{C_{1}}{C_{3}}+c_{0}\right) }{\left( b_{0}\frac{C_{1}}{C_{3}%
}+a_{0}\right) \left( d_{0}\frac{C_{1}}{C_{3}}+c_{0}\right) }.
\end{equation*}%
Note that if $\frac{C_{1}\lambda _{1}}{C_{3}}-\left( b_{0}\frac{C_{1}}{C_{3}}%
+a_{0}\right) \left( d_{0}\frac{C_{1}}{C_{3}}+c_{0}\right) =0$, then $%
K_{1}\left( n_{0}\right) \rightarrow 0$ as $n\rightarrow \infty $. That is
to say, (\ref{s1}) is convergent. Since $L_{1}=L_{2}$, (\ref{s2}) is
convergent, too. In this case, the proof of item (iii) follows from (\ref{t1}%
)-(\ref{t4}) and Theorem \ref{T}.
\end{proof}



\end{document}